\documentclass[twoside,reqno]{amsart}
\usepackage{amsmath,amscd,amssymb}
\usepackage{verbatim}
\usepackage{epsfig}

\newcommand{\End}{\operatorname{End}}

\newcommand{\Id}{\operatorname{Id}}

\newcommand{\tr}{\operatorname{tr}}
\newcommand{\Z}{\mathbb{Z}}

\newcommand{\C}{\mathbb{C}}
\newcommand{\N}{\mathbb{N}}
\newcommand{\sll}{\mathfrak{sl}}
\newcommand{\slt}{\sll(2)}

\newcommand{\g}{sl_2}

\newcommand{\halfroot}{\ensuremath{\epsilon}}
\newcommand{\Usplit}{\ensuremath{\stackrel{\circ}{U}_\epsilon(sl_2)}}
\newcommand{\Usplitg}{\ensuremath{\stackrel{\circ}{U}_\epsilon(\mathfrak{g})}}

\newcommand{\cat}{{\mathcal{C}_\epsilon}}

\newcommand{\ep}{\epsilon}
\newcommand{\RR}{\mathcal{R}}

\newtheorem{definition}{Definition}
\newtheorem{theorem}[definition]{Theorem}

\newtheorem{proposition}[definition]{Proposition}
\newtheorem{lemma}[definition]{Lemma}
\newtheorem{remark}[definition]{Remark}

\newcommand{\epsh}[2]
         {\begin{array}{c} \hspace{-1.0mm}
        \raisebox{-4pt}{\epsfig{figure=#1.eps,height=#2}}
        \hspace{-1.0mm}
        \end{array}}
\newcommand{\e}{\epsilon}
\newcommand{\io}{\iota}

\begin{document}
\let\co=\comment \let\endco=\endcomment

\title[]{On invariants of graphs related to quantum $\slt$ at roots of unity}
\author{Nathan Geer}
\address{Max-Plank Institute f\"{u}r Mathematik\\
Vivatgasse 7\\
53111 Bonn, Germany\\
and\\
School of Mathematics\\
  Georgia Institute of Technology\\
  Atlanta, GA 30332-0160, USA} \email{geer@math.gatech.edu}
    \author{Nicolai Reshetikhin}
\address{Department of Mathematics\\
University of California, Berkeley\\
970 Evans Hall \#3840\\
Berkeley, CA 94720-3840, USA\\
and\\
KDV Institute for Mathematics\\
Universiteit van Amsterdam, \\
Plantage Muidergracht 24\\
1018 TV, Amsterdam, The Netherlands}
\email{reshetik@math.berkeley.edu}
  \date{\today}

\begin{abstract}
We show how to define invariants of graphs related to quantum
$sl_2$ when the graph has more then
one connected component and components are colored by
blocks of representations with zero quantum dimensions.
\end{abstract}

\maketitle
\setcounter{tocdepth}{1}

\section*{Introduction}
\begin{co}
I have used the environment {co} to automatically
show/hide comments written in this block form.
To hind such comments just change the lines
\let\co=\comment \let\endco=\endcomment
\end{co}

The invariants of oriented linked framed graphs colored by
finite dimensional representations of quantized universal enveloping algebras were constructed in \cite{RT} for non-zero quantum dimensions.  In this construction the edges of the graph are colored by finite dimensional representations. At each vertex
a total ordering of adjacent edges which agrees with counter-clock wise cyclic order defined by the framing is fixed. A vertex
with adjacent edges colored by representations $V_1, \dots, V_n$ (according to the total ordering) is colored by an element
from $Hom(\C,V_1^{\e_1}\otimes...\otimes V_n^{\e_n})$ of all $U_q(\g)$-
invariant vectors in the tensor product. Here $V^+=V, V^-=V^*$
where $V^*$ is the left dual to $V$, $\e=+$ if the edges is incoming and $\e=-$ if the edges is outgoing.

When the quantum dimension is zero, it is easy to define such invariants for knots when the center of
$U_q(\g)$ acts by scalars on the representation coloring the
knot. One should cut the knot (its diagram) at any place and
then to compute the invariant of the corresponding $(1,1)$
tangle. Because the center acts trivially at the coloring representation, this invariant will be equal to a operator of multiplication by a scalar. It is easy to show that the
scalar is the invariant of a knot (does not depend on where the
knot was cut). But this technique does not work for
links.

Recently invariants of links colored by such representations
were introduced in \cite{GPT} under the assumption that
these representations satisfy an additional property,
i.e. they are ambidextrous (see \cite{GPT} or section \ref{S:ambi} for
the precise definition). We will abbreviate this name here to ambi-modules.

The goal of this note is to present in details
how the construction involving ambi-modules works for $sl_2$.
We prove that all generic irreducible representations of $\Usplit$ are ambi-modules.

We also conjecture but we do not prove it here that
all generic irreducible representations of the $\Usplitg$
are ambi-modules  for any simple Lie algebra.

We would like to thank V. Turaev and B. Patureau for discussions.
This work was supported by the Danish National
Research Foundation through the
Niels Bohr initiative, the work of N.R. was also
supported by the NSF grant DMS-0601912. The work of N.G.
was supported by the NSF grant DMS-0706725. Both authors are grateful to the Aarhus University for the
hospitality.

\section{The split form of quantum $\g$ at a root of unity}

\subsection{Quantum $\g$}

Let $m,l$ be positive odd integers, $t=\exp(\frac{i\pi m}{2l})$, and $\epsilon=t^2$.

The algebra $U_\halfroot(\g)$ is the unital
$\C$-algebra generated by invertible $K$ and by $E, F$
with defining relations:
\begin{equation*}
  KE=\halfroot^2EK,  KF=\halfroot^{-2}FK,
  [E,F]=\frac{K-K^{-1}}{\halfroot-\halfroot^{-1}}.
\end{equation*}

Recall that it is a Hopf algebra with the comultiplication acting on generators as
\[
\Delta K=K\otimes K, \Delta E=E\otimes K+1\otimes E,
\Delta F=F\otimes 1+K^{-1}\otimes F
\]

The elements $K^l, F^l, E^l$ generate a central Hopf subalgebra,
and $E^l, F^l$ generate central Hopf ideal.
Denote by $U_\halfroot(\g)'$ its quotient by this central ideal.

\subsection{The split form of $U_\halfroot(\g)'$}

\begin{definition}
The split form of $U_\halfroot(\g)'$ is the algebra $\Usplit$ generated by $P_i, i\in \Z$ and by $E,F,x,x^{-1}$ such that $x$ is central and
 \begin{equation*}
   P_iP_j=\delta_{ij}P_i, P_iE=EP_{i-2},  P_iF=FP_{i+2},  \sum_{i\in \Z} P_i = 1,
 \end{equation*}
 \begin{equation*}
 EF-FE=\sum_{j\in \Z}\frac{x^2\e^j-x^{-2}\e^{-j}}{\halfroot - \halfroot^{-1}}P_j,
\end{equation*}

\end{definition}

The algebra $\Usplit$ is a Hopf algebra (in the category of its finite-dimensional representations) where the comultiplication $\Delta$, counit $\varepsilon$ and antipode $S$ morphisms are given by:
\begin{align*}
 \Delta(E)= & \sum_{j\in \Z}E\otimes x^2\e^jP_j + 1\otimes E, &  \Delta(F)= & F\otimes 1 + \sum_{j\in \Z}x^{-2}\e^{-j}P_j\otimes F,\\
 \Delta(P_k)=& \sum_{i,j\in\Z ;\; i+j=k}P_i\otimes P_j, & \Delta(x)=&x\otimes x,
\end{align*}
\begin{align*}
  \varepsilon(E)=&\varepsilon(F)=0, & \varepsilon(P_i)=& \delta_{i,0}, & \varepsilon(x)=&1,
\end{align*}
\begin{align*}
 S(E)=& -\sum_{j\in \Z}x^{-1}\e^{-j}P_j,& S(F)=& -\sum_{j\in \Z}x\e^jP_jLF, & S(P_i)=&P_{-i}, & S(x)=&x^{-1}.
\end{align*}

It is clear that the map $\iota : U_\halfroot(\g)\rightarrow \Usplit$ given by $K\mapsto \sum_{j\in \Z}x^2\e^j P_j$, $E\mapsto E$ and $F\mapsto F$ is an injective morphism of Hopf algebras.

Notice that $K^l\mapsto \sum_{j\in \Z} x^{2l}\e^{jl}P_j$ and that $K^{2l}\mapsto x^{4l}$.

\subsection{The $R$-matrix}

For $n\in\N$ we set $[n]=\frac{\halfroot^n-\halfroot^{-n}}{\e-\e^{-1}}$ and $[n]!=[n][n-1]\dots[1]$.
Define
$$R_0=\sum_{i,j\in\Z}t^{ij}x^iP_j\otimes x^jP_i,$$
$$R_1=\sum_{n=0}^{l-1} \frac{(\e-\e^{-1})^{n}}{[n]!}E^n\otimes F^n,$$
and $R=R_0R_1\in \Usplit^{\otimes 2}$.

Now we will show that $R$ defines a quasitriangular structure on $\Usplit$.
\begin{lemma}\label{L:R0}
 The element $R_0$ satisfies the following identities.
 \begin{enumerate}
  \item \label{IL:A1} $(\Delta \otimes \Id)(R_0)=(R_0)_{13}(R_0)_{23}$,
  \item \label{IL:A2} $(\Id \otimes \Delta)(R_0)=(R_0)_{13}(R_0)_{12}$,
  \item \label{IL:A3} $R_0(E\otimes 1)=(E\otimes \io(K))R_0$, $R_0(1\otimes E)=(\io(K)\otimes E)R_0$,
  \item \label{IL:A4} $R_0(F\otimes 1)=(F\otimes \io(K)^{-1})R_0$, $R_0(1\otimes F)=(\io(K)^{-1}\otimes F)R_0$.
 \end{enumerate}
\end{lemma}
\begin{proof}
We will prove the first identity in Part \eqref{IL:A3}, the other identities follow similarly.  We have
 \begin{align*}
 R_0(E\otimes 1)&= \sum_{i,j\in\Z}t^{ij}x^iP_jE\otimes x^j P_i\\
 &=\sum_{i,j\in\Z}t^{ij}x^iEP_{j-2}\otimes x^j P_i\\
 &=\sum_{i,j\in\Z}t^{ij}x^iEP_{j}\otimes x^{j}x^2\halfroot^i P_i\\
 &=(E\otimes \io(K))R_0.
\end{align*}
\end{proof}
\begin{theorem}\label{T:quasitri}
 The pair $(\Usplit,R)$ is a quasitriangular Hopf algebra.
\end{theorem}
\begin{proof}
Lemma \ref{L:R0} implies that $R_0$ has the same commutator relation with $E$ and $F$ as the element $\exp(\frac{h}{4}(H\otimes H))$ in the $\C[[h]]$-algebra $U_h(\g)$. Thus, we conclude
\begin{align*}
 (\Delta \otimes \Id)(R)&=R_{13}R_{23}, & (\Id \otimes \Delta)(R)&=R_{13}R_{12},
\end{align*}
and
$$ \Delta^{op}(a)=R\Delta(a)R^{-1}, \;\; \text{ for } a\in \Usplit $$
where $\Delta^{op}=\tau \circ \Delta$ and $\tau$ is the
permutation $a\otimes b\mapsto b\otimes a$.
\end{proof}

Let $R=\sum s_i \otimes t_i$ and define $u=\sum S(t_i)s_i$.  Then following
 \cite{Dr} we have
\begin{align}\label{E:Ru}
  \epsilon(u)&=1, & \Delta(u)&=\left(R_{21}R\right)^{-1}(u\otimes u), &
  S^2(a)&=uau^{-1}
\end{align}
for all $a\in\Usplit$.

A direct computation shows that for all $a\in \Usplit$ we
have
\begin{equation}\label{E:S2K}
 S^2(a)=\io(K)a\io(K)^{-1}.
\end{equation}
Let $\tilde{\theta}=u\io(K)^{-1}=\io(K)^{-1}u$. Comparing this formula with the formula for the $S^2$ involving $u$ we see that $\tilde{\theta}$ is central.

\begin{lemma} \label{L:theta}The element $\tilde{\theta}$ satisfies the relations:
 \begin{align*}
 \epsilon(\tilde{\theta})&=1, & \Delta(\tilde{\theta})&=\left(R_{21}R\right)^{-1}(\tilde{\theta}\otimes\tilde{\theta}), & S(\tilde{\theta})&=\io(K)^{2l}\tilde{\theta}.
 \end{align*}
\end{lemma}
\begin{proof}
The first two relations follow from \eqref{E:Ru} and the definition of the counit and coproduct of $\io(K)$.
The last relation can be proven by direct computation of $S(u\io(K)^{-1})$ and of $u\io(K)^{-1}$. This computation was done essentially in \cite{Oh}.
One can also prove the last identity by computing how $S(\tilde{\theta})$ and $\tilde{\theta}$ act on generic irreducible
modules. It is easy to see that on these modules the identity holds. On the other hand our algebra is finitely generated and
is finite dimensional over the center. This implies that
the identity $S(\tilde{\theta})=\io(K)^{2l}\tilde{\theta}$
holds not only for generic points but also for special points
(where $x$ acts as $4l$-th root of unity).
\end{proof}

\subsection{The automorphism $\phi$}

\begin{proposition} \begin{enumerate}
\item The mapping $\phi: \Usplit\to \Usplit$ acting on generators as
\[
\phi(P_i)=P_{i+2}, \phi(x)=x\e, \phi(E)=E, \phi(F)=F
\]
extends uniquely to an algebra automorphism.
\item The set of fixed points form a Hopf subalgebra generated
by $x\sum_{i\in \Z}t^iP_i, x^{2l}, E, F$.
\end{enumerate}
\end{proposition}
\begin{proof}
The first part is clear from defining relations, the
second follows immediately from the formula for $\phi$.
\end{proof}

Note that despite the fact that fixed points of $\phi$
form a Hopf subalgebra, it is not a Hopf algebra automorphism.

The automorphism $\phi$ acts on the $R$-matrix as follows:
\[
(\phi\otimes id)(R)=(1\otimes x^{-2})R, \ \ (id\otimes \phi)(R)=(x^{-2}\otimes 1)(R)
\]

\begin{proposition}The automorphism $\phi$ acts on
$\tilde{\theta}$ as follows:
\[
\phi(\tilde{\theta})=x^{4}\tilde{\theta}
\]
\end{proposition}
\begin{proof} Let $R=\sum_i s_i\otimes t_i$, then $R^{-1}=\sum_i
s_i\otimes S(t_i)$. Combining the way how $\phi$ acts
on the $R$-matrix with the fact that it is an algebra automorphism we obtain $(\phi\otimes \phi)(R^{-1})=(x^2\otimes x^2)R^{-1}$.
Therefore $\phi(\sum_i S(t_i)s_i)=x^4\sum_iS(t_i)s_i$. thus, $\phi(u)=x^4u$. Since $\phi(\io(K))=\io(K)$, this proves the proposition.
\end{proof}

Define $\varepsilon=\sum_{j\in \Z}(-1)^j P_j$. It is central and
unipotent. It is also easy to see that $\phi$ acts trivially on it:
\[
\phi(\varepsilon)=\varepsilon
\]

\subsection{The extended $R$-matrix}
Consider the algebra $A_\ep=\Usplit\otimes_\C \C[\tau, \tau^{-1}]$. The element $\tau$ is central in this algebra.

\begin{remark} The motivation for this construction is
 the formal power series version of the algebra $\Usplit$,
 when instead of having Laurent polynomials in $x$ we have formal powers series z :
\[
x=\exp(\frac{imz\pi}{2l}), \ \ \tau=\exp(\frac{imz^2\pi}{4l})
\]
\end{remark}

Assume that the action of the comultiplication on $\tau$ is symmetric, i.e. $\Delta^{op}(\tau)=\Delta(\tau)$.
Define
\[
\rho=\Delta(\tau)\tau^{-1}\otimes\tau^{-1}
\]

Extend the action of the automorphism $\phi$ on $A_\ep$ as:
\[
\phi(\tau)=\ep x^2\tau, \ \ (\phi\otimes id)(\rho)=(1\otimes x^2)\rho
\]

Define the extended $R$-matrix as
\[
\RR=\rho R
\]

It is easy to see that $\RR$ defines a quasitriangular structure on $A_\ep$. It is also easy to see that
\[
(id\otimes \phi)(\RR)=(\phi\otimes id)(\RR)=\RR
\]

Define
\[
\theta_1=\tau^{-2}\io(K)^l\tilde{\theta}, \ \ \theta_2=\tau^{-2}\io(K)^l\varepsilon\tilde{\theta}
\]
Each of these elements satisfies the identities:
\begin{align*}
 \epsilon(\theta)&=1, & \Delta(\theta)&=\left(\RR_{21}\RR\right)^{-1}(\theta\otimes\theta), & S(\theta)&=\theta, & \phi(\theta)=\theta.
 \end{align*}

\section{The category $C_\ep$}

\subsection{Simple modules}
Here we will focus on simple modules over the
$\Usplit$, $U_\e(\g)'$, and $A_\epsilon$.
We will say $a\in \C^*$ is {\it generic} if $a^{4l}\neq 1$.

The following statement is a variation on a well known fact:
\begin{proposition} For each generic $a\in \C^*$ and $k\in \Z$ there exists unique simple $\Usplit$-module $V(a,k)$ with highest weight vector $v_k$ such that
\begin{align*}
 Ev_k&=0, & xv_{k}&=av_{k} & P_jv_{k}&=\delta_{j,k}v_{k}, & V(a,k)=\sum_{i=0}^{l-1}\C F^{i}v_{k}.
\end{align*}
\end{proposition}

It is easy to compute the action of generators on the
weight basis $v_{k-2j}=F^jv_k$:
\begin{align*}
 xv_{k-2i}&=av_{k-2i} \\
 P_jv_{k-2i}&=\delta_{j,k-2i}v_{k-2i}, \; 0\geq i \geq l-1,\\
 Fv_{k-2i}&=v_{k-2i-2}, 1\leq i \leq l-1 \\ Ev_{k-2i}&=\left(\frac{a^2\halfroot^{k+i+1}-a^{-2}\halfroot^{-k-1-i}}{\halfroot -\halfroot^{-1}}\right)\left(\frac{\halfroot^{i} -\halfroot^{-i}}{\halfroot -\halfroot^{-1}}\right)v_{k-2i+2}, 0\leq i\leq l-1
\end{align*}

The homomorphism $\io: U_\e(\g)'\to \Usplit$ defines
on $V(a,k)$ the structure of a $U_\e(\g)'$-module.

\begin{proposition} The $U_\e(\g)'$-modules $V(a,k)$ are irreducible, and in addition,  $V(a,k)$ and $V(a\e,k+1)$ are
isomorphic as $U_\ep'$- modules.
\end{proposition}

In particular central element $\varepsilon$ acts on $V(a,k)$
as $(-1)^k$.

Note that the modules $V(a,k)$ and $V(a\e,k+1)$
are not isomorphic as $\Usplit$-modules.

For $z\in \C$ we can always define the $A_\ep$-module structure 
on the representation $V(e^{\frac{imz\pi}{2l}},k)$ of $\Usplit$ by defining $\tau v=\exp(\frac{imz^2\pi}{4l})v$ for any $v\in V(e^{\frac{imz\pi}{2l}},k)$. It is clear that this defines an  irreducible  $A_\ep$-module.

\subsection{The category $C_\ep$}

{\it Objects} of the category $C_\ep$ are finite
dimensional $\Usplit$-modules $(V, \pi_V: A_\ep\to End(V))$
on which $x$
acts as a multiplication by a scalar $\pi_V(x)v=\exp(\frac{imz\pi}{2l})v$. Here $v\in V$, $m,l$ are as above and $z\in \C$.
The central element $\tau$ on such module acts by multiplication on $\exp(\frac{imz^2\pi}{4l})$.

{\it Morphisms} between two such modules are all
$\Usplit$-invariant linear maps.

This category is monoidal because it is a category of finite dimensional modules over a Hopf algebra.
It is a rigid monoidal category with the left
dual modules defined as
usual $(V^*, \pi_{V^*}=\pi^*_V\circ S)$ where $\pi^*_V(a)$
is the dual linear map to $\pi_V(a)$ and with usual injection and evaluation morphisms:
\begin{align*}
i_{V} :& \C \rightarrow V\otimes V^{*}, \text{ given by }
1 \mapsto \sum
 e_i\otimes e^i, \\
 e_{V}: & V^*\otimes V\rightarrow \C, \text{ given by }
  f\otimes w \mapsto f(w)
\end{align*}

It is easy to see that $C_\ep$ is a braided category with
the commutativity morphism $c=\{c_{V,W}\}$ where $c_{V,W}:V\otimes W \rightarrow
W \otimes V$ given by $v\otimes w \mapsto \tau(R(v\otimes w))$.  It is also a ribbon category with the ribbon morphisms (twists): $\theta_V:V\rightarrow V, v\mapsto \theta^{-1}v$.

In this category the braiding and the ribbon structure
agree with isomorphisms of modules induced by $\phi$.

The objects of this category are semisimple for generic
$z$, i.e. when $mz$ it is not an integer.

The linear mapping $v\to \theta^{-1}uv$ is a
an isomorphism of representations $V\to V^{**}$. Recall that the
quantum (functorial) dimension of $V$ is defined
as the composition mapping $\C\to V\otimes V^*\to V^{**}\otimes V^*\to \C$, or, as $tr_V(\theta^{-1}u)=tr_V(\tau^2\io(K)^{1-l})$.  It is clear that the
quantum dimension of any generic representation is zero.

\section{Invariants of links}

\subsection{Ambi-elements in a ribbon category}\label{S:ambi}
Here we will recall some results and definitions from \cite{GPT}.

Recall that in a ribbon category there is a natural notion of
a trace of an endomorphism of an object.  If $f: V\to V$, its
trace is
\[
\tr_V(f)=e_{V^*} \circ(\mu_V\otimes id_{V^*})\circ(f\otimes id_{V^*})\circ  i_V
\]
where $\mu_V:V\to V^{**}$ is the isomorphism between $V$ and $V^{**}$ determined by the braiding and the ribbon structure.
For the category of $\stackrel{\circ}{U}_\epsilon(sl_2)$-modules, $\mu_V=\pi_V(\iota(K)^{1-l})$.

We will use the following notations.  If $V$ is simple and $f:V\to V$ is a morphism, by definition $f=c(f) id_V$ for some $c(f)\in \kappa$ where $\kappa$ is the base field for our category.
We will assume $\kappa=\C$ and will use the graphical
notation for $c(f)$ shown on Fig. \ref{cent-elem}.

\begin{figure}
  \begin{center}
  $ \put(12,-2){{\Large $f$}}  \epsh{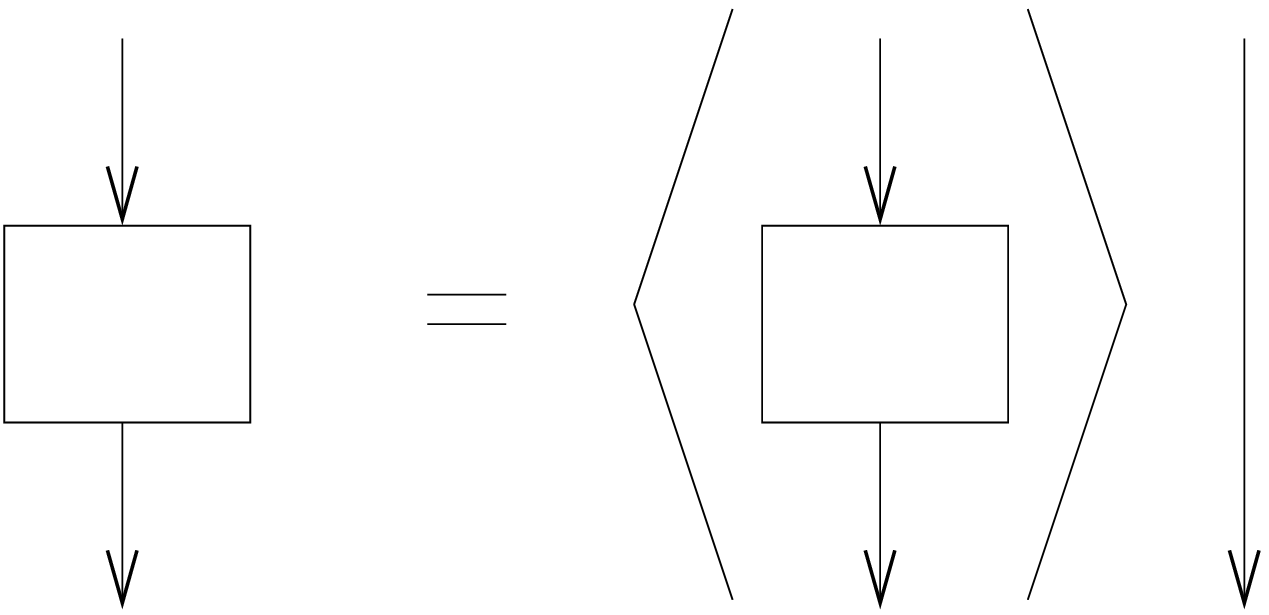}{14ex} \put(-43,-2){{\Large $f$}}$
   \caption{}
    \label{cent-elem}
  \end{center}
\end{figure}

Define $S'(U,V)=(id_V\otimes \tr_U)(c_{U,V}c_{V,U})$ where $c_{V,U}:V\otimes U\to U\otimes V$ is the commutativity constraint in the category. This element can be written graphically as on
Fig. \ref{S}.

\begin{figure}
  \begin{center}
$
\text{{\Large $S'(U,V)$}}=\epsh{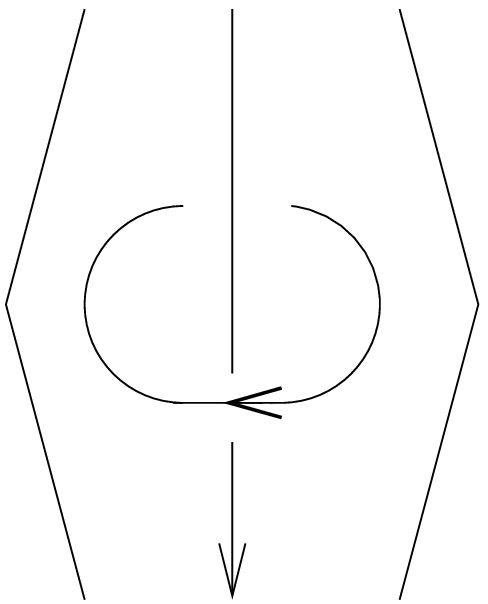}{14ex}  \put(-37,23){{\large $V$}} \put(-16,11){{\large $U$}}$
   \caption{}
    \label{S}
  \end{center}
\end{figure}

For any $f:V\to V$ and simple $V$ we have the identity shown  on Fig. \ref{cent-S}.

\begin{figure}
  \begin{center}
$ \put(9,0){{\Large $f$}}  \put(90,0){{\Large $f$}}  \put(40,22){{\large $V$}} \put(53,-10){{\large $U$}} \epsh{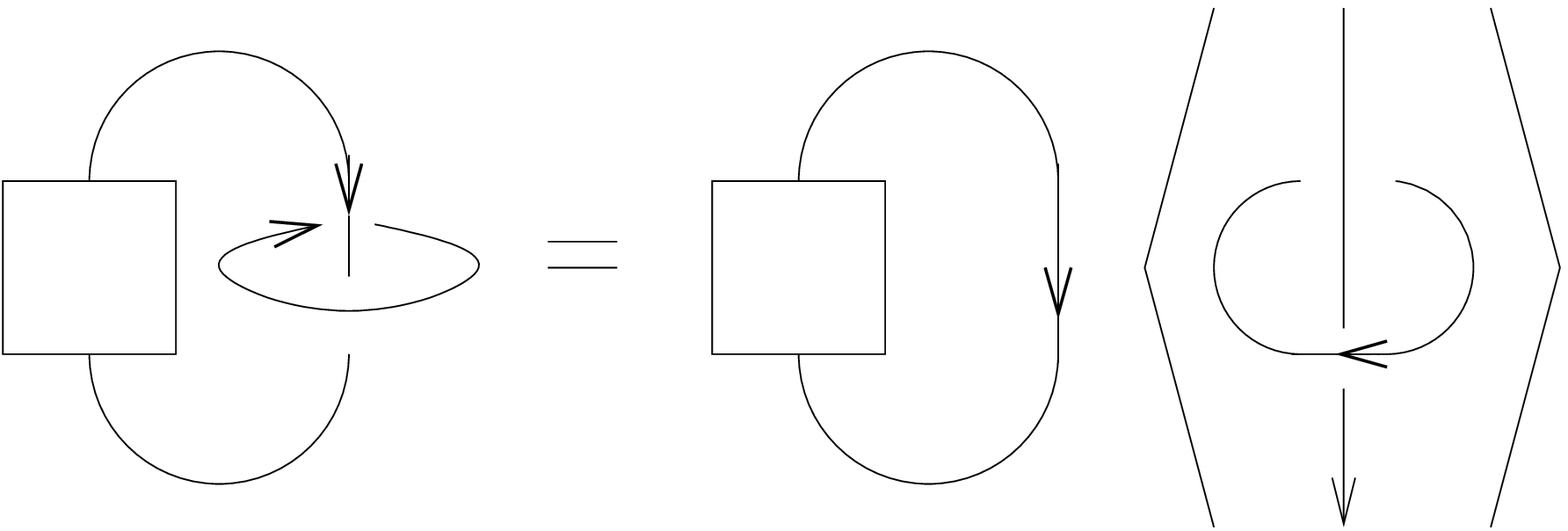}{14ex}\put(-62,20){{\large $V$}}   \put(-37,23){{\large $V$}} \put(-16,11){{\large $U$}}$
   \caption{}
    \label{cent-S}
  \end{center}
\end{figure}

Recall that an object $V$ is called {\it ambidextrous} if
$(id_V\otimes \tr_V)(f)=(\tr_V\otimes id_V)(f)$ for any $f:V^{\otimes 2}\to V^{\otimes 2}$ (see also Fig. \ref{F:ambi}). Here we will call such an object an {\it ambi-object}.

\begin{figure}
  \begin{center}
  $ \put(7,0){{\large $V$}}   \put(43,24){{\large $V$}}  \put(102,0){{\Large $f$}}  \put(37,0){{\Large $f$}}  \epsh{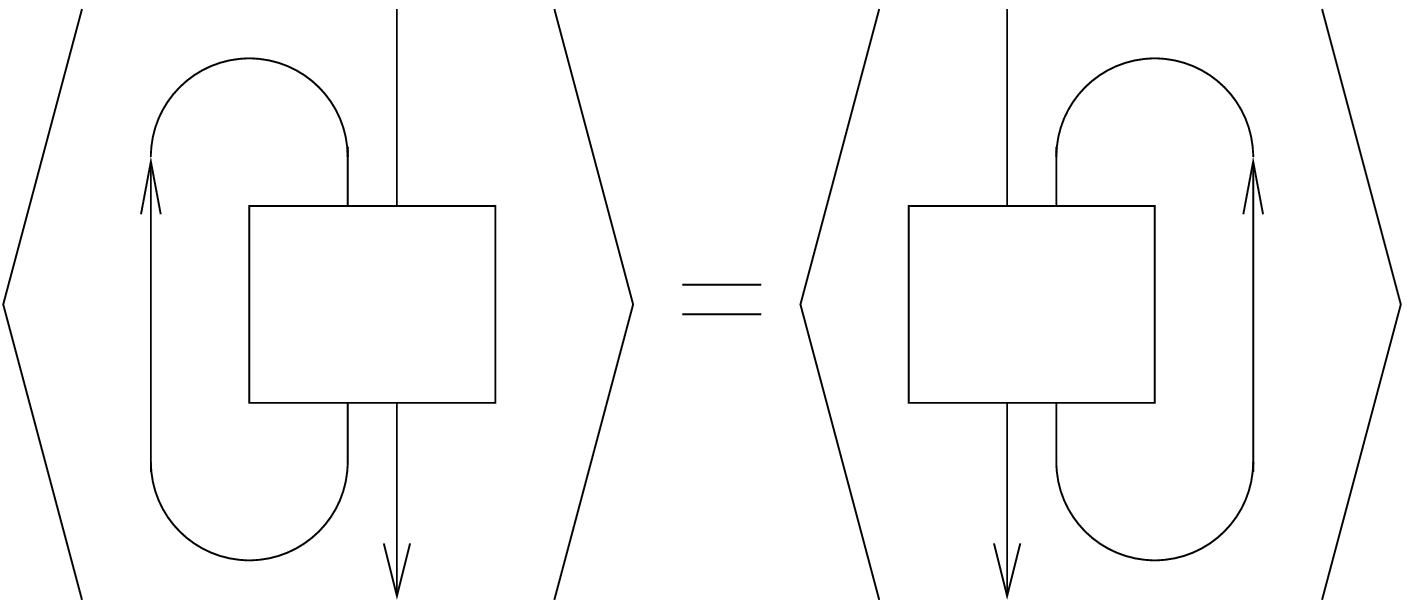}{14ex}  \put(-16,0){{\large $V$}}   \put(-52,24){{\large $V$}} $
   \caption{}
    \label{F:ambi}
  \end{center}
\end{figure}

The following observation is one of the key tools from \cite{GPT}.
Let $U$ be a simple ambi-object.
Then for any two objects $V_1$ and $V_2$ and any morphism $f:V_1\otimes V_2\rightarrow V_1\otimes V_2$ the following identity holds:
\begin{equation}\label{s-iden}
S'(U^*,V_1^*)S'(V_2,U)c((\tr_{V_1}\otimes id_{V_2})(f))=
S'(V_1,U)S'(U^*,V_2^*)c((id_{V_1}\otimes \tr_{V_2})(f))
\end{equation}
It follows from the ambi-identity for the tangle from Fig. \ref{ambi-proof}.  Here notice that $S'(V^*,W^*)=S'(V,W)$.

\begin{figure}
  \begin{center}
  $\put(62,33){{\large $V_1$}} \put(79,33){{\large $V_2$}} \put(14,7){{\large $U$}} \put(72,0){{\LARGE $f$}}
\epsh{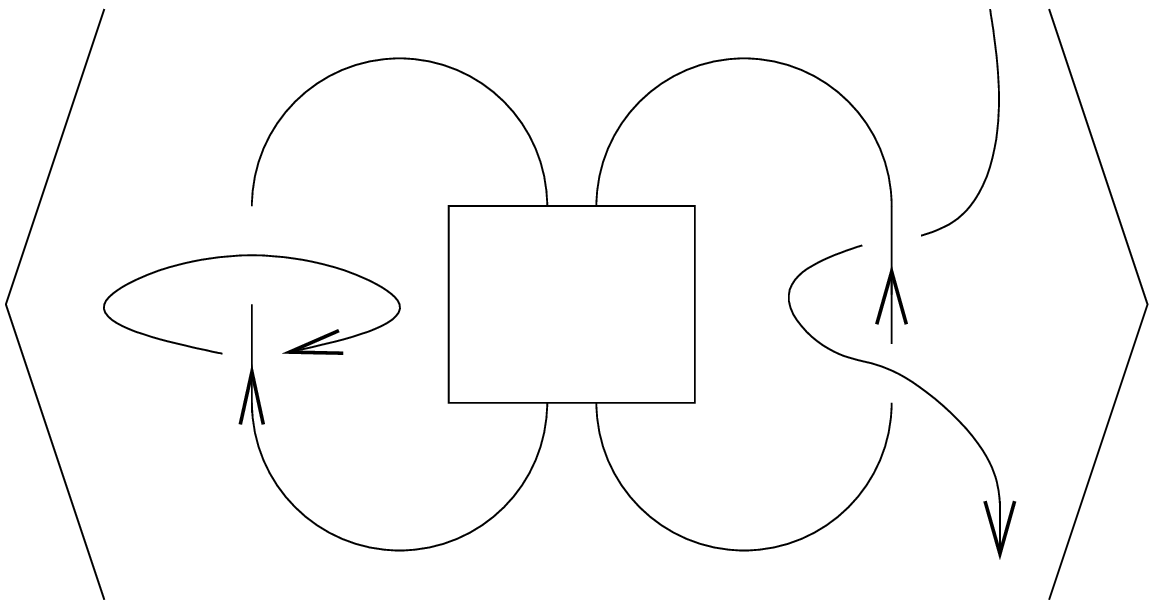}{18ex}  \put(-21,24){{\large $U$}}
  =\put(65,33){{\large $V_1$}} \put(83,33){{\large $V_2$}} \put(15,24){{\large $U$}} \put(75,-2){{\LARGE $f$}}
    \epsh{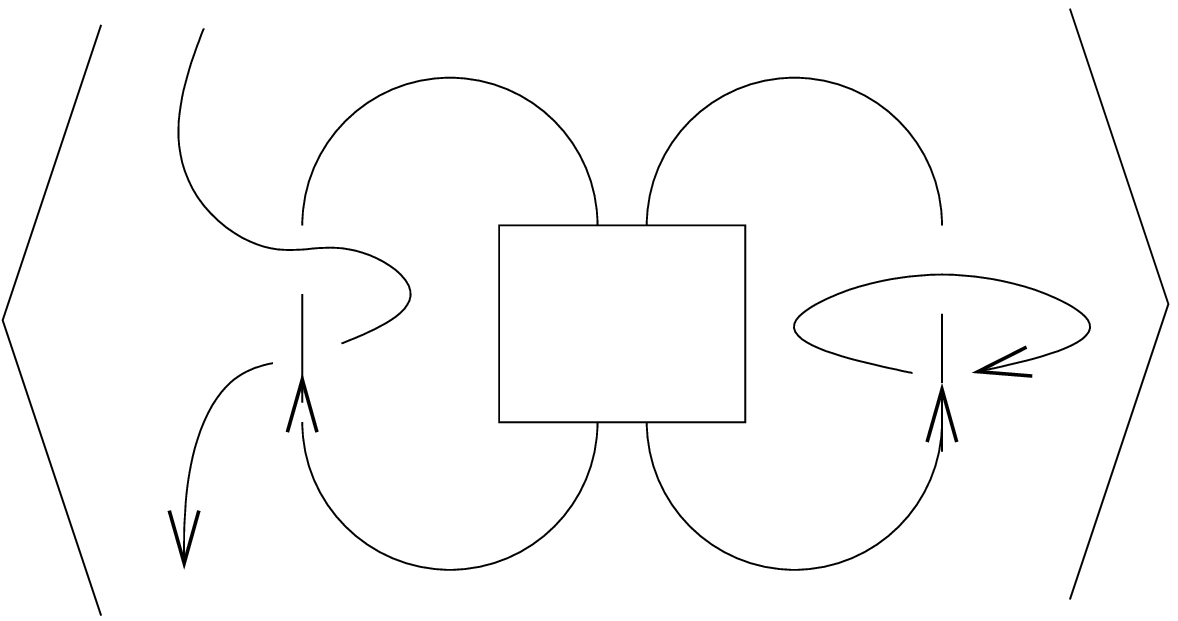}{18ex} \put(-18,7){{\large $U$}}
   $
   \caption{}
    \label{ambi-proof}
  \end{center}
\end{figure}

We have the following corollary of the identity (\ref{s-iden}).
\begin{proposition}\label{prop-GPT}
\begin{itemize}
\item If $J$ is a simple ambi-object of the category $\C$,
then all objects $U \in A(J)=\{W|S'(W,J)\neq 0, S'(J,W)\neq 0, W \mbox{ is simple}\}$ are also ambi-objects.

\item
Let  $T_V$ be a $(1,1)$-tangle colored by elements of $\cat$, whose open component is colored by an element $V\in A(J)$.  Then the expression in Fig. \ref{inv-link} is an invariant of the link obtained by the closure of $T_V$. In this expression $U$ is any ambi-element from $A(J)$.
\end{itemize}
\end{proposition}

\begin{figure}
  \begin{center}
$\put(22,-3){{\Large $T_V$}}
\epsh{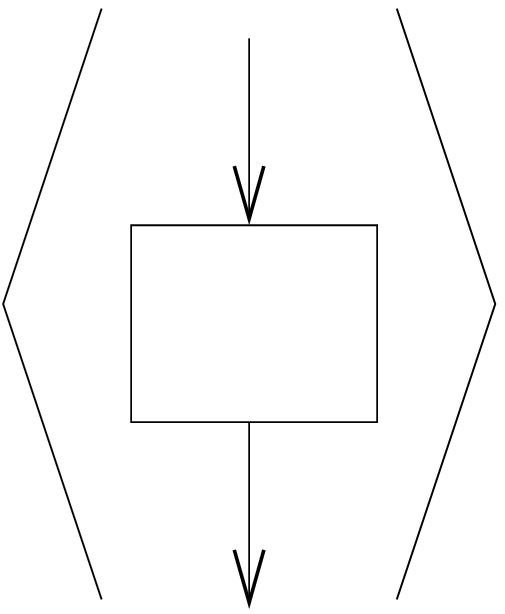}{14ex}  \put(-24,20){{\large $V$}} \put(-24,-23){{\large $V$}} \text{{\LARGE
$\frac{S'(V,U)}{S'(U,V)}$}}=\text{{\Large inv$(\widehat{T}_V)$}}$
   \caption{$\widehat{T}_V$ is the closer of $T_V$.}
    \label{inv-link}
  \end{center}
\end{figure}

\subsection{Ambi-modules over $\Usplit$}
\subsubsection{} As it was pointed out at the end of the previous section all modules $V(a,k)$ have zero quantum dimension.
Here we will prove that they are all ambi-objects in the category
$\Usplit-mod$.

\begin{lemma}\label{L:ambi}
The module $V(a,k)$ is an ambi-object for all $k\in \Z$
and $a^{8l}\neq 1$.
\end{lemma}
\begin{proof}
It is easy to prove that for all $a\in \C\setminus \{a^{8l}=1\}$, and $k\in \Z$
$$V(a,k)\otimes V(a,k)\cong V(2a,2k)\oplus  V(2a,2k-2)\oplus...\oplus V(2a,2k-2l+2).$$
Since this decomposition is a direct sum of non-isomorphic simple modules the algebra  $\End(V(a,k)\otimes V(a,k))$ is commutative.   Now let $f \in\End(V \otimes V)$, then
  $(id_V\otimes \tr_V) (f)=(\tr_V\otimes id_V)(c_{V,V}^{-1}\circ f \circ c_{V,V}).$
But $c_{V,V}$ commutes with $\End(V\otimes V)$ and so
$c_{V,V}^{-1}\circ f \circ c_{V,V}=f$. Here we set $V=V(a,k)$.
This finishes the proof.
\end{proof}

The condition $a^{8l}\neq 1$ can be relaxed to $a^{4l}\neq 1$.
\begin{lemma}\label{L:S'}
Let $V(a,k)$ and $V(b,m)$ be two irreducible $\Usplit$-modules described in the previous section.  Then
\begin{equation*}S'(V(a,k),V(b,m))
=\halfroot^{(k+1-l)(m+1-l)}b^{2k-2l+2}a^{2m-2l+2}(-1)^{m+l-1}
 \left[ \frac{b^{2l} - b^{-2l}}{\halfroot^{m+1-l}b^2 - (\halfroot^{m+1-l}b^2)^{-1}}\right]
\end{equation*}
\end{lemma}
\begin{proof}
The proof is a straightforward computation. We want to
compute the action of $(id_V\otimes \tr_U)((\pi_V\otimes\pi_U)(\sigma(R)R))$
on $V$. Let us apply to the highest weight vector. Then
\[
S'((V(a,k),V(b,m))v_{m}=(id_{V(b,m)}\otimes \tr_{V(a,k)})((\pi_{V(b,m)}\otimes\pi_{V(a,k)})(\sigma(R_0)R_0))v_m
\]
\[
=\halfroot^{k(m+1-l)}b^{2k}a^{2m+2-2l}\sum_{i=0}^{l-1}\halfroot^{-2i(m+1-l)}b^{-4i} v_m
\]
Then one should sum up the geometric progression.
\end{proof}

Now, note that $S'(V(a,k),V(b,m))\neq 0$ unless $b^{4l}=1$ when $(\halfroot^{m+1-l}b^2)^2\neq 1$. In the later case the representation $V(b,m)$ is reducible. We assume that this is not the case.

Thus, we have proven the following theorem.
\begin{theorem}
All irreducible representations $V(a,k)$ of $\Usplit$
with $a^{4l}\neq 1$ are ambi-modules.
\end{theorem}

\linespread{1}

\vfill


\begin{thebibliography}{99}
\bibitem{Dr} V. Drinfeld -
	 {Almost cocommutative Hopf algebras.} (Russian)
	 \emph{Algebra i Analiz} \textbf{1} (1989), no. 2, 30--46; translation in
	 \emph{Leningrad Math. J.} \textbf{1} (1990), no. 2, 321--342

\bibitem{GPT} N. Geer, B. Patureau-Mirand, V. Turaev - {Modified quantum
    dimensions and re-normalized link invariants.}  Preprint
  (math.QA/0711.4229).

 \bibitem{Kas} C. Kassel,
		{\em Quantum groups,}
		Springer-Verlag CTM 155 (1994)

\bibitem{Oh}T. Ohtsuki, {\em Quantum Invariants. A study of knots and 3-manifolds.} Series on Knots and Everything, {\bf 29}. World Scientific Publishing Co., Inc., River Edge, NL, 2002.

\bibitem{RT}N. Reshetikhin, V. Turaev, {\em Ribbon graphs and their invariants derived from quantum groups}, Comm. Math. Phys. Volume 127, Number 1 (1990), 1-26.

\bibitem{Ro} M. Rosso, {\em Quantum groups at roots of 1 and tangle invariants}, Topological and geometrical methods in field theory (Turku, 1991), 347-358, World Scientific Publishing Co., Inc., River Edge, NL, 1992.

\bibitem{Tu} V.G. Turaev - {\em Quantum invariants of knots and 3-manifolds.}
  de Gruyter Studies in Mathematics, 18. Walter de Gruyter \& Co., Berlin,
  (1994).
\end{thebibliography}
\end{document}